 \documentclass[12pt,a4paper, oneside, bold,secthm,seceqn,amsthm,ussrhead,reqno]{article}
\usepackage[utf8]{inputenc}
\usepackage[english]{babel}
\usepackage[symbol]{footmisc}
\usepackage{amssymb,amsmath,amsthm,amsfonts,xcolor,enumerate,hyperref, comment,longtable, cleveref}
\usepackage{verbatim}
\usepackage{times}
\usepackage{cite}
\usepackage{pdflscape}
\usepackage{ulem}
\usepackage[mathcal]{euscript}
\usepackage{tikz}
\usepackage{hyperref}
 
\usepackage{cancel}
\usepackage{stmaryrd}
 
\usepackage{amsfonts}
\usepackage{amssymb}
\usepackage{times}
\usepackage{xcolor}
\usepackage{tkz-graph}
\usepackage{url}
\usepackage{float}
\usepackage{tasks}
\usepackage{array}

\usepackage{cite}
\usepackage{hyperref}

 \usepackage{fancyhdr} 
\usepackage{amsfonts}
\usepackage{amsmath}
\usepackage{eurosym}
\usepackage{geometry}

\newcommand{\red}{}

\usepackage{caption,booktabs}

\newtheorem{lemma}{Lemma} 
\newtheorem{theo}[lemma]{Theorem}
\newtheorem{prop}[lemma]{Proposition}

\newtheorem{cor}[lemma]{Corollary}

\theoremstyle{definition}

\newtheorem{defin}[lemma]{Definition}

\newtheorem{remark}[lemma]{Remark}

\newcommand{\Leib}{{\rm Leib}}

 \usepackage{fancyhdr} 
\numberwithin{equation}{section}

\newenvironment{eq}{\begin{equation}}{\end{equation}}

\newcommand{\Char}{\mathop{\rm char}}

\newcommand{\FF}{\mathbb{F}}

\newcommand{\CC}{\mathbb{C}}

\newcommand{\NN}{\mathbb{N}}

\newcommand{\algA}{\mathcal{A}}

\newcommand{\algL}{\mathcal{L}}
\newcommand{\algV}{\mathcal{V}}
\newcommand{\algX}[1]{{\rm alg}_{\FF}\{X\}_{#1}}

\newcommand{\calcP}{{\mathcal P}}

\newcommand{\ncl}{{\rm ncl}}

\newcommand{\diag}[1]{{\rm diag}\{#1\}}

\newcommand{\ML}[1]{M_{\rm L}^{(i)}} 
\newcommand{\MR}[1]{M_{\rm R}^{(i)}} 

\newcommand{\MdL}[2]{M_{\rm L}^{(#1,#2)}} 
\newcommand{\MdR}[2]{M_{\rm R}^{(#1,#2)}} 

\newcommand{\al}{\alpha}
\newcommand{\be}{\beta}
\newcommand{\ga}{\gamma}
\newcommand{\la}{\lambda}
\newcommand{\de}{\delta}

\newcommand{\LA}{\langle}
\newcommand{\RA}{\rangle}

\newcommand{\un}[1]{{\underline{#1}} }
\newcommand{\alg}{\mathop{\rm alg}}
\newcommand{\mdeg}{\mathop{\rm mdeg}}

\newcommand{\tr}{\mathop{\rm tr}}
\newcommand{\Tr}{\mathop{\rm Tr}}

\newcommand{\Aut}{\mathop{\rm Aut}}

\newcommand{\GL}{{\rm GL}}
\newcommand{\Ann}{{\rm Ann}}

\newcommand{\matr}[4]{\left(\begin{array}{cc}
#1 & #2 \\
#3 & #4 \\
\end{array}\right)}

\newcommand{\matrThree}[9]{\left(\begin{array}{ccc}
#1 & #2 & #3\\
#4 & #5 & #6 \\
#7 & #8 & #9 \\
\end{array}\right)}

\newcommand{\OO}{\mathbf{O}}

\usepackage{fouriernc}
\begin{document}

\noindent{\Large
Polynomial invariants for $3$-dimensional Leibniz algebras}
 \footnote{
The work is supported by the FCT 2023.08031.CEECIND, UIDB/00212/2020, UIDP/00212/2020, and FAPESP 2023/17918-2}

 \bigskip

\begin{center}

 {\bf
   Ivan Kaygorodov\footnote{CMA-UBI, University of  Beira Interior, Covilh\~{a}, Portugal;  \ kaygorodov.ivan@gmail.com} 
   \footnote{Moscow Center for Fundamental and Applied Mathematics, Moscow, Russia}
   \&
   Artem Lopatin\footnote{ 
State University of Campinas, Campinas,  Brazil;\ dr.artem.lopatin@gmail.com 
}}

\end{center}

\ 

\noindent {\bf Abstract:}
{\it For each $3$-dimensional non-Lie Leibniz algebra over the complex numbers, we describe the algebra of polynomial invariants and determine its group of automorphisms. As a consequence, we establish that any two non-nilpotent $3$-dimensional non-Lie Leibniz algebras can be distinguished by the traces of degrees $\leqslant 2$ and by the dimensions of their automorphism groups.}

 \bigskip 


\bigskip 

 \
 

	 \bigskip

\ 

\


\tableofcontents

\section{Introduction}\label{section_intro}

\subsection{Algebra of invariants}\label{section_alg_inv}   Assume that $\FF$ is an algebraically closed field of an arbitrary characteristic $\Char{\FF}$. All vector spaces and algebras are over $\FF$.     

Consider an algebra $\algA$ is of dimension $n$, i.e., $\algA$ is a vector space with a bilinear multiplication, and a subgroup $G$ of the group of all automorphisms $\Aut(\algA)\leqslant \GL_n$ of $\algA$. The group $G$ naturally acts on $\algA^m = \algA \oplus \cdots \oplus \algA$ ($m$ times) diagonally: $g\un{a} = (g\cdot a_1,\ldots,g\cdot a_m)$ for each $g\in G$ and $\un{a}=(a_1,\ldots,a_m)$ from $\algA^m$. The coordinate ring of the vector space $\algA^m$ is a polynomial (i.e. commutative and associative) algebra 
$$\FF[\algA^m] = \FF[x_{ri}\,|\, 1\leqslant r\leqslant m, \; 1\leqslant i\leqslant n].$$
Fix some basis $\{e_1,\ldots,e_n\}$ for $\algA$, and for every $a\in\algA$ denote by $(a)_i$ the $i-{\rm th}$ coordinate of $a$ with respect the given basis. 
We can consider elements of $\FF[\algA^m]$  as polynomial functions $\algA^m\to\FF$, since   $x_{ri}$ can be interpreted as a function $\algA^m\to\FF$ given by $x_{ri}(\un{a}) = (a_r)_i$. 
The action of $G$ on $\algA^m$ induces the action of the coordinate ring $\FF[\algA^m]$ as follows: $(g\cdot f)(\un{a}) = f(g^{-1}\cdot\un{a})$ for all $g\in G$, $f\in \FF[\algA^m]$ and $\un{a}\in\algA^m$. The algebra of {\it $G$-invariants of the $m$-tuple of the algebra  $\algA$} is
$$\FF[\algA^m]^{G}=\{f\in \FF[\algA^m]\,|\,gf=f \text{ for all }g\in G\},$$
or, equivalently, 
$$\FF[\algA^m]^{G}=\{f\in \FF[\algA^m]\,|\,f(g\un{a})=f(\un{a}) \text{ for all }g\in G,\; \un{a}\in \algA^m\}.$$

For short, this algebra of $\Aut(\algA)$-invariants of the $m$-tuple of $\algA$ is called the algebra of {\it invariants of the $m$-tuple of $\algA$}, and we denote it by
$$I_m(\algA):=\FF[\algA^m]^{\Aut(\algA)}.$$%
It is well known that the so-called operator traces (see Section~\ref{section_tr} for the details) are $\Aut(\algA)$-invariants. Denote by $\Tr(\algA)_m$ the subalgebra of $I_m(\algA)$ generated by all operator traces together with $1$. We say that the {\it Artin--Procesi--Iltyakov Equality} holds for $\algA^m$ if 
$$I_m(\algA) = \Tr(\algA)_m.$$%
The Artin--Procesi--Iltyakov Equality is known to be valid for $\algA^m$ with $m\geqslant 1$ for the following algebras:
\begin{enumerate}
\item[$\bullet$] $\algA=M_n$ is the algebra of all $n\times n$ matrices over $\FF$, in case $\Char{\FF}=0$ or $\Char{\FF}>n$. Generators for the algebra of invariants $I_m(M_n)$ were independently described in~\cite{Sibirskii68} and~\cite{Procesi76} in the case of characteristic zero, and in~\cite{Donkin92a} in the case of positive characteristic. Minimal generating sets for $I_m(M_2)$ and $I_m(M_3)$ were obtained in~\cite{Procesi84, DKZ02,  Lopatin_Comm1, Lopatin_Sib, Lopatin_Comm2} and for $I_m(M_4)$ and $I_m(M_5)$ with small values of $m$, in~\cite{Teranishi86, Drensky_Sadikova_4x4,Djokovic07}, assuming $\Char\FF=0$. 

\item[$\bullet$] $\algA=\OO$ is the octonion algebra, in case $\Char{\FF}\neq2$. Generators for the algebra of invariants $I_m(\OO)$ were constructed in~\cite{schwarz1988} over the field of complex numbers $\CC$ and in~\cite{zubkov2018} over an arbitrary infinite field of odd characteristic. In case $\Char{\FF}\neq2$, a minimal generating set was constructed~\cite{LZ_2}, using the classification of pairs of octonions from~\cite{LZ_1}.

\item[$\bullet$] $\algA={\mathbb A}$ is the split Albert algebra, i.e., the exceptional simple Jordan algebra of $3\times 3$ Hermitian matrices over $\OO$ with the symmetric multiplication $a\circ b = (ab+ba)/2$, in case $\Char{\FF}=0$ and $m\in\{1,2\}$ (see~\cite{Iltyakov_1995, Polikarpov_1991}).

{\red
\item[$\bullet$] $\algA$ is a two-dimensional simple algebra with a non-trivial automorphism group, in case $\Char{\FF}=0$ (see~\cite{Alvarez_Lopatin_2025}).
}
\end{enumerate}

\noindent{}More details about the Artin--Procesi--Iltyakov Equality can be found in~\cite{Alvarez_Lopatin_2025}.

\subsection{Leibniz algebras}\label{section_Leibniz} 
 An algebra $\algL$ is called a (right) {\it   Leibniz  algebra}  if it satisfies the identity
$$(xy)z=(xz)y+x(yz).$$ 

Leibniz algebras present a "non antisymmetric" generalization of Lie algebras.
It was at first introduced by A. Bloh~\cite{Bloh_1965} and then
independently by J.-L. Loday~\cite{L93a}. 
Recently, they appeared in many geometric and physics applications (see, for example, 
\cite{bonez,   kotov20,   KW01, strow20} and references therein).
A systematic study of algebraic properties of Leibniz algebras was started by J.-L. Loday and T. Pirashvili~\cite{L93}.
So, several classical theorems from Lie algebras theory have been extended
to the Leibniz algebras case;
many classification results regarding nilpotent, solvable, etc. Leibniz algebras
have been obtained 
(see, for example, \cite{FKS_2025, AKS, deMello_Souza_2023, CK,FW,KM,kky22, MS, TX} and references therein);
for results on algebraic and geometric classifications of low dimensional Leibniz algebras see \cite{k23,KKP} and references therein.

Given a Leibniz algebra $\algL$, we denote by $\Leib(\algL)$ the $\FF$-span of $a^2$ for all $a\in\algL$. Since
$$yx^2 = 0 \;\text{ and }\; x^2 y = (x^2 + y)^2 - (x^2)^2 - y^2,$$%
we have that $\Leib(\algL)$ is an ideal of $\algL$. Note that $\Leib(\algL)$ is non-zero if and only if $\algL$ is not a Lie algebra. On the other hand, if $\Leib(\algL)=\algL$, then $\algL$ has zero multiplication. Therefore, non-Lie Leibniz algebra is never simple.

\subsection{Results}\label{section_results} In this section, we assume that $\FF = \CC$ is the field of complex numbers. Denote by ${\bf L}_3$ the family of all $3$-dimensional non-Lie Leibniz algebras over the field $\FF$. The classification of algebras from ${\bf L}_3$ modulo the action of the automorphism group was independently obtained by:
{\red
\begin{enumerate}
    \item[$\bullet$] J.M.~Casas, M.A.~Insua, M.~Ladra, and S.~Ladra~\cite{Casas_Insua_Ladra_Ladra_2012_LAA}, by means of a computer program;
    \item[$\bullet$] I.M.~Rikhsiboev and I.S.~Rakhimov~\cite{Rikhsiboev_Rakhimov_2012_AIP}, without the use of a computer. Unfortunately, this classification is not complete (see the proof of Theorem~\ref{theo_Leib3} for details).
\end{enumerate}
In our paper, we complete the classification given in~\cite{Rikhsiboev_Rakhimov_2012_AIP} in Theorem~\ref{theo_Leib3}, and we prove its equivalence with the classification obtained in~\cite{Casas_Insua_Ladra_Ladra_2012_LAA} (see again the proof of Theorem~\ref{theo_Leib3}).}

In Section~\ref{section_prelim}, we provide the key definitions of generic elements and operator traces. In Section~\ref{section_gen}, we present explicit formulas for calculating certain operator traces (see Proposition~\ref{prop_trace}) and describe the reduction of the general case to the multilinear case in characteristic zero (see Proposition~\ref{prop_char0}).

Assume $\algL\in {\bf L}_3$. In Sections~\ref{section_simple2dim}--\ref{section_inv} we assume that $\FF=\CC$ is the field of complex numbers. In Section~\ref{section_simple2dim} we explicitly describe the group $\Aut(\algL)$ of all automorphisms of $\algL$ (see Theorem~\ref{theo_aut}).  Then in Section~\ref{section_invariants} we obtain a generating set for the algebra of invariants $I_m(\algL)$ (see Theorem~\ref{theo_main}). As a consequence, we describe when the Artin--Procesi--Iltyakov Equality holds for $\algL$ (see Corollary~\ref{cor_API}). In Corollary~\ref{cor_nilp} we obtain $I_m(\algL)=\FF$ for all $m\geqslant1$ if and only if $\algL$ is nilpotent.
As a consequence, in Corollary~\ref{cor_separ} we show that any two non-nilpotent algebras from ${\bf L}_3$ can be distinguished by means of 
\begin{enumerate}
\item[$\bullet$] the traces of degrees $\leqslant2$, 

\item[$\bullet$] the dimensions of its groups of automorphisms.
\end{enumerate}%

Note that the problem of distinguishing low-dimensional algebras using their algebraic properties has recently been studied. Specifically, in \cite{diniz2024isomorphism}, it was shown that any two-dimensional Jordan algebras over a finite field $\FF$ with odd characteristic are isomorphic if and only if they satisfy the same polynomial identities. In \cite{PIs_Novikov_dim2_2025}, it was established that polynomial identities distinguish non-associative two-dimensional Novikov algebras over the complex numbers.

\subsection{Notations}\label{section_notations} 

A monomial $w= x_{r_1,i_1}\cdots x_{r_k,i_k}$ from $\FF[\algA^m]$ has {\it multidegree} $\mdeg(w)=(\de_1,\ldots,\de_m)\in\NN^m$, where $\de_r$ is the number of letters of $w$ lying in the set $\{x_{r1},\ldots, x_{rn}\}$ and $\NN=\{0,1,2,\ldots\}$. As an example, $\mdeg(x_{31}x_{32}x_{21}x_{12})=(1,1,2)$ for $m=3$. For short, we denote the multidegree $(1,\ldots,1)\in\NN^m$ by $1^m$. If $f\in\FF[\algA^m]$ is a linear combination of monomials of the same multidegree $\un{d}$, then we say that $f$ is $\NN^m$-{\it homogeneous} of multidegree $\un{d}$. In other words, we have defined the $\NN^m$-grading of $\FF[\algA^m]$ by multidegrees. Since $\FF$ is infinite, the algebra of invariants $\FF[\algA^m]^G$ also has the $\NN^m$-grading by multidegrees. An $\NN$-homogeneous element $f\in\FF[\algA^m]$ of multidegree $(d_1,\ldots,d_m)$ with $d_1,\ldots,d_m\in\{0,1\}$ is called {\it multilinear}. For short, we denote by  $x_{r_1,i_1}\cdots\widehat{x_{r_l,i_l}}\cdots x_{r_k,i_k}$ the monomial $x_{r_1,i_1}\cdots x_{r_{l-1},i_{l-1}} x_{r_{l+1},i_{l+1}} \cdots x_{r_k,i_k}$ from   $\FF[\algA^m]$, where $1\leqslant l\leqslant k$. Given a subset $S\subset\algA$, we denote by $\alg\{S\}$ the subalgebra of $\algA$  (without unity in general) generated by $S$.
We write $\FF^{\times}$ for the set of non-zero elements of $\FF$.


If there exists $k\geqslant1$ such that for every $a_1,\ldots,a_k\in \algA$ each product of $a_1,\ldots,a_k$ is zero in $\algA$, then the algebra $\algA$ is called {\it nilpotent} and the least possible $k\geqslant1$ with the above property is called the {\it nilpotency class} $\ncl(\algA)$ for $\algA$; otherwise, we write $\ncl(\algA)=\infty$.

\section{Preliminaries on invariants}\label{section_prelim}

\subsection{Generic elements} 
To explicitly define the action of $G$ on $\FF[\algA^m]$ consider  the algebra $\widehat{\algA}_m=\algA \otimes_{\FF} \FF[\algA^m]$, which is a $G$-module, where the multiplication and the $G$-action are defined as follows: $(a\otimes f)(b\otimes h)=ab \otimes fh$ and $g\bullet (a\otimes f) = g\cdot a \otimes f$ for all $g\in G$, $a,b\in\algA$ and $f,h\in\FF[\algA^m]$. Define by 
$$X_r = \left(\begin{array}{c}
x_{r1} \\
\vdots \\
x_{rn} \\
\end{array}\right):=e_1 \otimes x_{r1} + \cdots + e_n \otimes x_{rn}$$
the {\it generic elements} of $\widehat{\algA}_m$, where $1\leqslant r\leqslant m$. In particular, we have 
$$g\bullet X_r = 
\left(\begin{array}{c}
gx_{r1} \\
\vdots \\
gx_{rn} \\
\end{array}\right) 
\in\widehat{\algA}_m.$$

For $g\in G\leqslant \GL_n$ we write $[g]$ for the corresponding $n\times n$ matrix. By straightforward calculations we can see that 
\begin{eq}\label{eq_action}
g\bullet X_r = [g]^{-1}X_r.
\end{eq}%

Denote by $\algX{m}$ the subalgebra of $\widehat{\algA}_m$  generated by the generic elements $X_1,\ldots,X_m$ and the unity $1$. Any product of the generic elements is called a word of $\algX{m}$. Since $G\leqslant \Aut(\algA)$, we  have that 
\begin{eq}\label{eq_action_prod}
g\bullet (FH) = (g\bullet F)(g\bullet H)
\end{eq}%
for all $F,H$ from $\algX{m}$.

\subsection{Operator traces}\label{section_tr} 

For $a\in\algA$ denote by $L_a:\algA \to\algA$ and  $R_a:\algA \to\algA$  the operators of the left and right multiplication by $a$, respectively. Then define the {\it left operator trace} $\tr_{\rm L}:\algA \to \FF$ by $\tr_{\rm L}(a)=\tr(L_a)$ and the {\it right operator trace} $\tr_{\rm R}:\algA \to \FF$ by $\tr_{\rm R}(a)=\tr(R_a)$. 

To expand these constructions, we denote by $\FF\LA \chi_0,\ldots,\chi_m\RA$ the absolutely free unital algebra in letters $\chi_0,\chi_1,\ldots,\chi_m$ and for $f\in \FF\LA \chi_0,\ldots,\chi_m\RA$, $\un{a}=(a_0,\ldots,a_m)\in\algA^{m+1}$ define $f(\un{a})\in\algA$ as the result of substitutions $\chi_0\to a_0,\, \chi_1\to a_1,\ldots,\chi_m\to a_m$ in $f$. For $h\in\FF\LA \chi_0,\ldots,\chi_m\RA$ denote by $L_h,R_h:\FF\LA \chi_0,\ldots,\chi_m\RA \to \FF\LA \chi_0,\ldots,\chi_m\RA$ the operators of the left and right multiplication by $h$, respectively.   As usually, the composition of maps $P_1,P_2:\algA\to \algA$ is denoted by $P_1\circ P_2 (a)=P_1(P_2(a))$, $a\in\algA$. Similar notation we use for composition of maps $P_1,P_2:\FF\LA \chi_0,\ldots,\chi_m\RA \to \FF\LA \chi_0,\ldots,\chi_m\RA$. For a symbol $P\in\{L,R\}$, we write 
\begin{enumerate}
\item[$\bullet$] $P_a:\algA\to\algA$ for the operator of left or right multiplication by $a\in\algA$;

\item[$\bullet$] $P_h:\FF\LA \chi_0,\ldots,\chi_m\RA \to \FF\LA \chi_0,\ldots,\chi_m\RA$ for the operator of left or right multiplication by $h\in\FF\LA \chi_0,\ldots,\chi_m\RA$.
\end{enumerate}

\noindent{}The following definition of the operator trace generalizes definitions of the left and right operator traces.

\begin{defin}\label{def1}
Assume that $m\geqslant1$ and $h\in \FF\LA \chi_0,\ldots,\chi_m\RA$ is homogeneous of degree 1 in $\chi_{0}$, i.e., each monomial of $h$ contains $\chi_{0}$ exactly once. Then 
\begin{enumerate}
\item[$\bullet$] for every $\un{a}\in\algA^m$ define the linear operator $h(\,\cdot\,,\un{a}):\algA\to\algA$  by the following equality: $h(\,\cdot\,,\un{a})(b) = h(b,a_1,\ldots, a_m)$ for all $b\in\algA$;

\item[$\bullet$]  define {\it operator trace} as follows: $\tr(h):\algA^m \to \FF$, where $\tr(h)(\un{a})=\tr(h(\,\cdot\,,\un{a}))$ for all $\un{a}\in\algA^m$.
\end{enumerate}
\end{defin}

As an example, $\tr(\chi_0)=n$, $\tr_{\rm L}=\tr(\chi_1\chi_0)$, and $\tr_{\rm R}=\tr(\chi_0\chi_1)$. 

\begin{remark}\label{remark_key} Assume that $h\in \FF\LA \chi_0,\ldots,\chi_m\RA$ is homogeneous of degree 1 in $\chi_{0}$.
\begin{enumerate}
\item[(a)] It is easy to see that $\tr(h) \in \FF[\algA^m]$.

\item[(b)] It is well known that $\tr(h)$ is an invariant from $I_m(\algA)$ (for example, see Lemma~3.1 from~\cite{Alvarez_Lopatin_2025}). 

\item[(c)] If $h$ has multidegree $(1,\de_1,\ldots,\de_m)\in\NN^{m+1}$, then $\tr(h)$ is $\NN^m$-homogeneous of multidegree $(\de_1,\ldots,\de_m)$.
\end{enumerate}
\end{remark}

\begin{defin}\label{def3} Denote by $\Tr(\algA)_m$ the subalgebra of $I_m(\algA)$ generated by $1$ and operator traces $\tr(h)$ for all that $h\in \FF\LA \chi_0,\ldots,\chi_m\RA$ homogeneous of degree 1 in $\chi_{0}$.
\end{defin}





\section{Generators for algebras of invariants}\label{section_gen}

\subsection{Traces}\label{section_trace_alg}

 Assume that the tableau of multiplication of the algebra $\algA$ is $M=(M_{ij})_{1\leqslant i,j\leqslant n}$, i.e., $e_ie_j=M_{ij}$ in $\algA$. Denote
$$M_{ij}=\sum_{l=1}^n M_{ijl} e_l$$
for some $ M_{ijl}\in\FF$. For every $1\leqslant i\leqslant n$ consider the following $n\times n$ matrices over $\FF$: 
$$\ML{i}=(M_{ijl})_{1\leqslant l,j\leqslant n} \;\; \text{ and  }\;\;
\MR{i}=(M_{jil})_{1\leqslant l,j\leqslant n}.$$

\begin{prop}[Proposition 3.2 of~\cite{Alvarez_Lopatin_2025}]\label{prop_trace}
Assume that $m\geqslant1$ and a homogeneous monomial $h\in \FF\LA \chi_0,\ldots,\chi_m\RA$ of degree 1 in $\chi_{0}$ satisfies the equality $h= P^1_{\chi_{r_1}} \circ \cdots \circ P^k_{\chi_{r_k}} (\chi_0)$ for some symbols $P^1,\ldots,P^k\in\{L,R\}$, $1\leqslant r_1,\ldots,r_k\leqslant m$ and $k>0$. Then 

$$\tr(h)=\sum_{1\leqslant i_1,\ldots,i_k\leqslant n} \tr\!\Big(M_{P^1}^{(i_1)} \cdots M_{P^k}^{(i_k)}\Big) x_{r_1,i_1} \cdots x_{r_k,i_k}.$$
\end{prop}

\noindent{}Proposition~\ref{prop_trace} implies that for all $1\leqslant r\leqslant m$ we have 
\begin{eq}\label{eq_trL_trR}
\tr(\chi_r\chi_0)=   \sum_{i,j=1}^n x_{ri} M_{ijj}
\;\; \text{ and } \;\; 
\tr(\chi_0\chi_r)= \sum_{i,j=1}^n x_{ri} M_{jij}.
\end{eq}

For every $1\leqslant i,i'\leqslant n$ consider the following $n\times n$ matrices over $\FF$: 
$$\MdL{i}{i'}=\Big( \sum_{t=1}^nM_{i i' t} M_{tjl} \Big)_{1\leqslant l,j\leqslant n} \;\; \text{ and  }\;\;
\MdR{i}{i'} = \Big( \sum_{t=1}^nM_{i i' t} M_{jtl} \Big)_{1\leqslant l,j\leqslant n}.$$

\begin{prop}[Proposition 3.3 of~\cite{Alvarez_Lopatin_2025}]\label{prop_trace_double}
Assume that $m\geqslant1$ and a homogeneous monomial $h\in \FF\LA \chi_0,\ldots,\chi_m\RA$ of degree 1 in $\chi_{0}$ satisfies the equality $h= P^1_{h_1} \circ \cdots \circ P^k_{h_k} (\chi_0)$ for some symbols $P^1,\ldots,P^k\in\{L,R\}$ and monomials $h_1,\ldots,h_k\in \FF\LA \chi_1,\ldots,\chi_m\RA$ of degree 1 or 2, where $k>0$. For every $1\leqslant q\leqslant k$ denote 
\begin{enumerate}
\item[$\bullet$] $\un{i}_q=i_q$ and $w(q,\un{i}_q) = x_{r_q,i_q}$, in case $h_q=\chi_{r_q}$ for some 
$1\leqslant r_q\leqslant m$;

\item[$\bullet$] $\un{i}_q=\{i_q,i'_q\}$, $M_{P^q}^{(\un{i}_q)}=M_{P^q}^{(i_q,i'_q)}$ and $w(q,\un{i}_q) = x_{r_q,i_q}x_{r'_q,i'_q}$, in case $h_q=\chi_{r_q}\chi_{r'_q}$ for some 
$1\leqslant r_q,r'_q\leqslant m$.
\end{enumerate}

Then 

$$\tr(h)=\sum_{1\leqslant \un{i}_1,\ldots,\un{i}_k\leqslant n} \tr\!\Big(M_{P^1}^{(\un{i}_1)} \cdots M_{P^k}^{(\un{i}_k)}\Big) w(1,\un{i}_1) \cdots w(k,\un{i}_k),$$
\noindent{}where the condition $1\leqslant\{i_q,i'_q\}\leqslant n$ stands for the condition $1\leqslant i_q,i'_q \leqslant n$.
\end{prop}

\subsection{Reduction to multilinear case}

In this section we assume that $G\leqslant  \GL_n$ and  $\FF=\CC$ is the field of complex numbers.

A vector $\un{r}=(r_1,\ldots,r_m)\in\NN^m$ satisfying $r_1+\cdots +r_m=t$ is called a {\it partition} of $t\in \NN$ in $m$ parts. Denote by $\calcP^{t}_m$ the set of all such partitions. Given a partition $\un{r} \in \calcP^{t}_m$, we define a function $\un{r}|\cdot|:\{1,\ldots,t\}\to \{1,\ldots,m\}$ by
$$\un{r}|l|=j \text{ if and only if } 
\sum_{1\leqslant i\leqslant j-1} r_i + 1\leqslant l \leqslant \sum_{1\leqslant i\leqslant j} r_i,$$
for all $1\leqslant l\leqslant t$. For a partition $\un{r} \in\calcP^{t}_m$ we also define a homomorphism of $\FF$-algebras as follows:
\begin{center}
    $\pi_{\un{r}}: \FF[\algA^t] \to \FF[\algA^m]$  by $x_{l,i}\to x_{\un{r}|l|,i}$ for all  $1\leqslant l\leqslant t$ and $1\leqslant i\leqslant n$.

\end{center}
The following proposition and lemma are well known and can easily be proven.

\begin{prop}[cf.~II.5, Theorem 2.5A of~\cite{Weyl_book}]\label{prop_char0}
The algebra of invariants $\FF[\algA^m]^G$ is generated by 
$$\{\pi_{\un{r}}(f)\,|\, f\in \FF[\algA^{t}]^G \text{ is multilinear},\; \un{r}\in \calcP^{t}_m,\; t\geqslant m\}.$$
\end{prop}

\begin{lemma}[cf.~Theorem 5.14 of~\cite{Bruns_Gubeladze_book_2009}]\label{lemma_diag}
If the group $G$ is diagonal, i.e., all elements of $G\leqslant \GL_n$ are diagonal matrices, then the algebra of invariants $\FF[\algA^m]^G$ is generated by some monomials from $\FF[\algA^m]$.
\end{lemma}


 The following lemma follows immediately from a partial case of Theorem 7.6 from~\cite{Alvarez_Lopatin_2025} for the algebra ${\mathbf A}_2$.

\begin{lemma}[see Theorem 7.6 of~\cite{Alvarez_Lopatin_2025}]\label{lemma_dim2}
Assume that $\algV$ is a two-dimensional vector space with a basis $\{e_1,e_2\}$ and let $G\leqslant \GL_2$ be the group of all matrices
$$\matr{1}{\al}{0}{1},$$
where $\al\in\FF$. As in Section~\ref{section_alg_inv}, the coordinate ring  $\FF[\algV^m]=\FF[x_{r1},x_{r2}\,|\,1\leqslant r\leqslant m]$ of $\algV^m$ is a $G$-module. Then the algebra of invariants $\FF[\algV^m]^G$ is generated by 
$$1, x_{12},\ldots,x_{m2},\; x_{r1}x_{s2} - x_{r2}x_{s1} \;\; (1\leqslant r<s\leqslant m).$$
\end{lemma}

\section{Automorphisms of $3$-dimensional  Leibniz algebras}\label{section_simple2dim}

We use notations for the tableau of multiplication of an algebra from Section~\ref{section_trace_alg}.

\begin{theo}\label{theo_Leib3} Each algebra $\algL\in{\bf L}_3$ is isomorphic to one and only one algebra from the following list, where algebras are given by their tableaux of multiplication with respect to some basis $\{e_1,e_2,e_3\}$:
\begin{longtable}{lcllcllcl}
$\algL_1$&$: $& $\matrThree{0}{0}{-2 e_1}{0}{e_1}{-e_2}{0}{e_2}{0},$ &
$\algL_2^{{\red \la\neq0}}$&$: $&$ \matrThree{0}{0}{\la e_1}{0}{0}{-e_2}{0}{e_2}{0},$ &
$\algL_3$&$: $&$ \matrThree{0}{0}{0}{0}{0}{-e_2}{0}{e_2}{e_1},$ \\
$\algL_4^{\la}$&$: $& $\matrThree{0}{0}{0}{0}{e_1}{e_1}{0}{0}{\la e_1},$&
$\algL_5$&$: $& $\matrThree{0}{0}{0}{0}{e_1}{0}{0}{0}{e_1},$ &
$\algL_6$&$: $&$ \matrThree{0}{0}{e_2}{0}{0}{e_1}{0}{0}{0},$ \\
$\algL_7^{\la}$&$:$ & $\matrThree{0}{0}{e_2}{0}{0}{\la e_1 + e_2}{0}{0}{0},$&
$\algL_8$&$:$ &$ \matrThree{0}{0}{e_2}{0}{0}{0}{0}{0}{e_1},$ &
$\algL_9$&$: $& $\matrThree{0}{0}{e_1 + e_2}{0}{0}{0}{0}{0}{e_1},$\\  
\red $\algL_{10}$ & $:$ & \red $\matrThree{0}{0}{e_1}{0}{0}{e_2}{0}{0}{0},$ &
\red $\algL_{11}$  & $: $& \red $\matrThree{0}{0}{0}{0}{0}{0}{0}{0}{e_1}.$ &\\ 
\end{longtable}
\noindent where $\la\in\FF$. 
\end{theo}
\begin{proof}
A classification of 3-dimensional non-Lie Leibniz algebras over complex numbers was obtained by Rikhsiboev, Rakhimov in Theorem~2 of~\cite{Rikhsiboev_Rakhimov_2012_AIP}, where it was claimed that any algebra from ${\bf L}_3$ belongs to the following list: $\algL_1$,  $\algL_2^{\la}$, $\algL_3$, $\algL_4^{\la}$, $\algL_5$, $\algL_6$, $\algL_7^{\la}$, $\algL_8$, $\algL_9$ with $\la\in\FF$. Unfortunately, this classification is not complete: 

\begin{enumerate}
\item[$\bullet$] The algebra $\algL_2^0$ is a Lie algebra, therefore, we should assume that $\la$ is non-zero in case $\algL_2^{\la}$. 

\item[$\bullet$] On page 361 of the proof of Theorem~2 of~\cite{Rikhsiboev_Rakhimov_2012_AIP} it was claimed that $\algL_{10}$ is isomorphic to $\algL_6$, which is not the case. Actually, as it follows from Corollary~\ref{cor_separ} (see below), $\algL_{10}$ is not isomorphic to any other algebra from the formulation of Theorem~\ref{theo_Leib3}.

\item[$\bullet$] The algebra $\algL_{11}$ was missed in Case 2.1 of the proof of Theorem 2 of~\cite{Rikhsiboev_Rakhimov_2012_AIP}. To show that $\algL_{11}$ is not isomorphic  to any other algebra from the formulation of Theorem~\ref{theo_Leib3}, we note that the only commutative algebras from the formulation of Theorem~\ref{theo_Leib3} are $\algL_{5}$ and $\algL_{11}$, but $\dim \Ann^{\rm R}(\algL_{5})=1$ and $\dim \Ann^{\rm R}(\algL_{11})=2$, where the right annihilator of an algebra $\algA$  is
$$\Ann^{\rm R}(\algA) =\{a\in \algA \,|\, ba=0 \text{ for all } b\in\algA\}.$$
\end{enumerate}

To complete the proof, we will prove that the classification from Theorem~\ref{theo_Leib3} is equivalent to the computer-based classification of 3-dimensional non-Lie Leibniz algebras over complex numbers by J.M. Casas, M.A. Insua, M. Ladra, S. Ladra~\cite{Casas_Insua_Ladra_Ladra_2012_LAA}.  For $\de\in\FF$, we denote by 
\begin{eq}\label{eq_casas}
\algL_{2a}^{\de},\; \algL_{2b},\; \algL_{2c},\; \algL_{2d},\; \algL_{2e}^{\de}\;  (\de\neq0),\; \algL_{2f},\; \algL_{2g},\; \algL_{3a}^{\de}\;  (\de\in\widetilde{\FF}),\; \algL_{3b},\; \algL_{3c},\;  \algL_{3d}
\end{eq}%
the 3-dimensional Leibniz algebras from page 3752 of~\cite{Casas_Insua_Ladra_Ladra_2012_LAA}, where 
$$\widetilde{\FF}= \{\al\in\CC\,|\, |\al| > 1\} \cup \{\al\in\CC \,|\, |\al| = 1,\; 0\leqslant{\rm arg}(\al)\leqslant \pi\}.$$%
\noindent{}Note that $\pm 1\in \widetilde{\FF}$ and for every $\al\in\FF\backslash \{0,\pm1\}$ either $\al$ or $\al^{-1}$ lies in $\widetilde{\FF}$.  In~\cite{Casas_Insua_Ladra_Ladra_2012_LAA} it was proven that any algebra from ${\bf L}_3$ is isomorphic to one and only one algebra from the list~(\ref{eq_casas}). The following isomorphisms complete the proof, where 
\begin{enumerate}
\item[$\bullet$] for each line, the algebra from the first column is isomorphic to the algebra from the second column; 

\item[$\bullet$] the tableau of multiplication of the algebra from the first column considered with respect to the ordered basis from the third column is equal to the tableau of multiplication of the algebra from the second column.
\end{enumerate}

\begin{longtable}{l|l|l}
\hline
$\algL_1$ & $\algL_{2g}$ & $\{e_1, e_1+e_2, -e_2-e_3\}$\\  

\hline   

$\algL_2^{\la\neq0}$ & $\algL_{2e}^{-\la}$ & $\{e_1, e_2, -e_3\}$\\  

\hline   

$\algL_{3}$ & $\algL_{2f}$ & $\{e_1, e_2, -e_3\}$\\  

\hline   

$\algL_{4}^{\la\neq0}$ & $\algL_{2a}^{\la}$ & $\{e_1, e_3, e_2\}$\\  

\hline   

$\algL_{4}^0$ & $\algL_{2a}^0$ & $\{e_1, e_3, e_2\}$\\  

\hline   

$\algL_{5}$ & $\algL_{2c}$ & $\{e_1, e_2, e_3\}$\\  

\hline   

$\algL_{6}$ & $\algL_{3a}^{-1}$ & $\{e_1+e_2, e_2-e_1, -e_3\}$\\  

\hline   

$\algL_{7}^{\la\neq -\frac{1}{4},0}$ & $\algL_{3a}^{\la'}$ & $\{\xi_1 e_1+e_2, \xi_2 e_1+e_2, \nu e_3\}$\\  

\hline   

$\algL_{7}^{-\frac{1}{4}}$ & $\algL_{3b}$ & $\{e_2, -\frac{e_1}{2}+e_2, 2e_3\}$\\  

\hline   

$\algL_{7}^0$ & $\algL_{2d}$ & $\{e_2, e_1-e_2, e_3\}$\\  

\hline   

$\algL_{8}$ & $\algL_{3c}$ & $\{e_1, e_2, e_3\}$\\  

\hline   

$\algL_{9}$ & $\algL_{3d}$ & $\{2e_1+e_2, 2e_1+2e_2, e_1+e_3\}$\\  

\hline   

$\algL_{10}$ & $\algL_{3a}^1$ & $\{e_1, e_2, e_3\}$\\  

\hline   

$\algL_{11}$ & $\algL_{2b}$ & $\{e_1, e_2, e_3\}$\\ 
\hline   
\end{longtable}%
\noindent{}Here, 
\begin{enumerate}
\item[$\bullet$] $\xi_1,\xi_2\in\FF$ are pairwise different roots of the equation $x^2+x-\la=0$, which are uniquely defined by the condition that $\la':=\frac{\xi_2}{\xi_1}$ belongs to $\widetilde{\FF}\backslash\{\pm1\}$; 

\item[$\bullet$] $\nu=\frac{2\xi_1+1}{\xi_1-2\la}$.
\end{enumerate}

It is straightforward to verify all cases but the case of $\algL=\algL_{7}^{\la}$ with $\la\in\FF\backslash\{-1/4,0\}$. Since the discriminant of the equation $x^2+x-\la=0$ is non-zero, then its roots $\xi_1,\xi_2$ are pairwise different. Note that 
\begin{eq}\label{eq_L7_0}
\xi_1 \xi_2 = -\la \;\text{ and }\; \xi_1+\xi_2=-1.
\end{eq}%

\noindent{}Since $\la\neq -1/4$, we have
\begin{eq}\label{eq_L7_1}
\xi_1\neq 2\la \;\text{ and }\; \xi_2\neq 2\la.
\end{eq}%

\noindent{}Since $\xi_1$ and $\xi_2$ are non-zero, $\la'$ is well-defined and $\la'\neq0$. Permuting $\xi_1$ and $\xi_2$, if necessary, we can assume that $\la'\in\widetilde{\FF}$. Obviously, if $\la'=-1$, then $\xi_1=\xi_2$; a contradiction. Similarly, if $\la'=1$, then $\la=-1/4$; a contradiction. Hence, $\la'\in\widetilde{\FF}\backslash\{\pm1\}$.

Formulas~(\ref{eq_L7_1}) imply that $\nu$ is well-defined and $\nu\neq0$. Hence,

$$v_1=\xi_1 e_1+e_2,\;\; v_2=\xi_2 e_1+e_2,\;\; v_3=\nu e_3$$%

\noindent{}is a basis for $\algL$. 

Compute the tableau of multiplication for $\algL$ with respect to the ordered basis $\{v_1,v_2,v_3\}$: 
\begin{enumerate}
\item[$\bullet$]
Using formulas~(\ref{eq_L7_0}) it is not difficult to see that $\nu\la=\la' \xi_1$ and $\nu(\xi_1+1)=\la'$. Hence, $$v_1v_3=\xi_1\nu e_2 + \nu (\la e_1+e_2) = \nu \la e_1 + \nu (\xi_1+1) e_2 = \la' \xi_1 e_1 + \la' e_2 = \la' v_1. $$%

\item[$\bullet$] Using formulas~(\ref{eq_L7_0}) it is not difficult to see that $\nu\la= \xi_2$ and $\nu(\xi_2+1)=1$. Hence, 
$$v_2v_3=\xi_2\nu e_2 + \nu (\la e_1+e_2) = \nu \la e_1 + \nu (\xi_2+1) e_2 =  \xi_2 e_1 + e_2 =  v_2. $$%

\item[$\bullet$] others products $v_i v_j$ are zero.
\end{enumerate}
Therefore, the tableau of multiplication for $\algL$ with respect to the ordered basis $\{v_1,v_2,v_3\}$ is
$$\matrThree{0}{0}{\la' v_1}{0}{0}{v_2}{0}{0}{0}.$$
Hence, $\algL\simeq \algL_{3a}^{\la'}$. The following claim, which can be proven by straightforward computations, completes the proof:

\begin{longtable}{l}
{\it for every $\la'\in\widetilde{\FF}\backslash\{\pm1\}$  there exists a unique  $\la\in\FF\backslash\{0,-1/4\}$ with $\la'=\frac{\xi_2}{\xi_1}$,} \\
{\it where $\xi_1,\xi_2\in\FF$ are pairwise different roots of the equation $x^2+x-\la=0$.}\\
\end{longtable}

\end{proof}

\begin{theo}\label{theo_aut}
The groups of automorphisms of algebras from ${\bf L}_3$ are the following set of matrices:
\begin{longtable}{lll}
$\algL_1$&$: $& $g=\matrThree{\al_5^2}{\al_5 \al_6}{\al_6^2/2}{0}{\al_5}{\al_6}{0}{0}{1} \text { for }\al_5\neq 0;$\\

$\algL_2^{{\red \la\neq 0}}$&$: $& $g=\matrThree{\al_1}{0}{0}{0}{\al_5}{\al_6}{0}{0}{1}  \text{ for } \al_1,\al_5\neq 0;$\\ 


$\algL_3$&$:$ &$ g=\matrThree{1}{0}{\al_3}{0}{\al_5}{\al_6}{0}{0}{1}  \text{ for } \al_5\neq0;$\\
$\algL_4^{\la\neq0}$&$: $& $g_1=\matrThree{\frac{\al_6^2 + \al_6 \al_9 + \la \al_9^2}{\la}}{\al_2}{\al_3}{0}{\frac{\al_6}{\la} + \al_9}{\al_6}{0}{-\frac{\al_6}{\la}}{\al_9}   \text{ for } \al_6\neq 0,\; \det(g_1)\neq0,$\\ 
 &&$ g_2=\matrThree{\al_5^2}{\al_2}{\al_3}{0}{\al_5}{0}{0}{0}{\al_5}   \text{ for } \al_5\neq 0 ;$\\
$\algL_4^{0}$&$: $& $g=\matrThree{\al_5 (\al_5 + \al_8)}{\al_2}{\al_3}{0}{\al_5}{0}{0}{\al_8}{\al_5 + \al_8}  \text{ for } \al_5+\al_8\neq0;$\\
$\algL_5$&$: $& $g_1=\matrThree{\al_5^2 + \al_6^2}{\al_2}{\al_3}{0}{\al_5}{\al_6}{0}{-\al_6}{\al_5}  \text{ for } a_5\neq0,\;\al_5^2+\al_6^2 \neq 0,$\\ 
 &&$  g_2=\matrThree{\al_5^2 + \al_6^2}{\al_2}{\al_3}{0}{\al_5}{\al_6}{0}{\al_6}{-\al_5}  \text{ for }a_5\neq0,\;\al_5^2+\al_6^2 \neq 0,$ \\
 &&  $g_3=\matrThree{\al_6^2}{\al_2}{\al_3}{0}{0}{\al_6}{0}{\pm\al_6}{0}   \text{ for } \al_6 \neq 0;$ \\
$\algL_6$&$: $& $g_1=\matrThree{\al_1}{\al_4}{0}{\al_4}{\al_1}{0}{0}{0}{1}  \text{ for }\al_1^2 \neq \al_4^2;$\\ 
 &&$ g_2=\matrThree{\al_1}{-\al_4}{0}{\al_4}{-\al_1}{0}{0}{0}{-1}   \text{ for } \al_1^2 \neq \al_4^2;$\\
$\algL_7^{\la\neq0}$&$: $&$ g=\matrThree{\al_1}{\la \al_4}{0}{\al_4}{\al_1 + \al_4}{0}{0}{0}{1}   \text{ for } \al_1^2+\al_1 \al_4-\la \al_4^2\neq 0   ;$\\  
$\algL_7^{0}$&$: $&$ g=\matrThree{\al_1}{0}{\al_3}{\al_4}{\al_1 + \al_4}{-\al_3}{0}{0}{1}   \text{ for } \al_1 (\al_1+\al_4)\neq 0;$\\ 
$\algL_8$&$: $& $g=\matrThree{\al_9^2}{0}{\al_3}{\al_3 \al_9}{\al_9^3}{\al_6}{0}{0}{\al_9}  \text{ for } \al_9\neq 0;$\\ 
$\algL_9$&$: $& $g=\matrThree{1 + \al_3}{0}{\al_3}{\al_3}{1}{\al_6}{0}{0}{1}   \text{ for }\al_3\neq-1;$    \\
\red $\algL_{10}$&$: $& \red $g=\matrThree{\al_1}{\al_2}{0}{\al_4}{\al_5}{0}{0}{0}{1}   \text{ for }\al_1\al_5-\al_2\al_4\neq0;$    \\
\red $\algL_{11}$&$: $& \red $g=\matrThree{\al_9^2}{\al_2}{\al_3}{0}{\al_5}{\al_6}{0}{0}{\al_9}   \text{ for }\al_5\al_9\neq0.$    
\end{longtable}
\noindent  Here $\al_1,\ldots,\al_9\in\FF$.
\end{theo}
\begin{proof}
Given $\algL\in{\bf L}_3$, consider an invertible linear map $g:\algL\to\algL$, given by a matrix
$$g=\matrThree{\al_1}{\al_2}{\al_3}{\al_4}{\al_5}{\al_6}{\al_7}{\al_8}{\al_9}$$
with respect to the basis $\{e_1,e_2,e_3\}$ of $\algL$ from Theorem~\ref{theo_Leib3}, where $\al_1,\ldots,\al_9\in\FF$. Then $g$ is an automorphism of $\algL$ if and only if for each $1\leqslant i,j\leqslant 3$ we have $g_{ij}:=g(e_ie_j)-g(e_i)g(e_j) = 0$. It is straightforward to verify that matrices $g,g_i$ from the formulation of this theorem are automorphisms. So, we assume that $g$ is an automorphism.

\medskip
\begin{enumerate}[{\bf 1.}]

\item Assume $\algL=\algL_1$. Consequently considering $g_{13}=0$, $g_{11}=0$, $g_{23}=0$ we obtain $\al_7 = 0$, $\al_4 = 0$, $\al_8 = 0$, respectively. 
Since $\det(g)=\al_1 \al_5 \al_9$, we have $\al_1\neq0$ and consequently considering $g_{13}=0$, $g_{22}=0$, $g_{23}=0$, $g_{33}=0$ we obtain $\al_9 = 1$, $\al_1 = \al_5^2$, $\al_2 = \al_5 \al_6$, $\al_3 = \al_6^2/2$, respectively. The required statement is proven.

\item Assume $\algL=\algL_2^{\la}$ for some {\red non-zero} $\la\in\FF$.  Consequently considering $g_{13}=0$, $g_{23}=0$, $g_{32}=0$ we obtain $\al_7 = 0$,  $\al_8 = 0$, $\al_2 = 0$, respectively. Since $\det(g)=\al_1 \al_5 \al_9\neq0$, we consequently consider $g_{31}=0$, $g_{33}=0$, $g_{13}=0$ to obtain $\al_4 = 0$, $\al_3 = 0$, $\al_9 = 1$, respectively.


\item Assume $\algL=\algL_3$. Consequently considering $g_{11}=0$, $g_{22}=0$, $g_{23}=0$,   we obtain $\al_7 = 0$, $\al_8 = 0$, $\al_2 = 0$, respectively. Since $\det(g)=\al_1 \al_5 \al_9\neq0$, we consequently consider $g_{13}=0$, $g_{23}=0$, $g_{33}=0$ to obtain $\al_4 = 0$, $\al_9 = 1$, $\al_1 = 1$, respectively. The required statement is proven.

\item Assume $\algL=\algL_4^{\la}$ for some $\la\in\FF$. Since $g_{22}=0$, we have  $\al_4 = 0$ and $\al_7 = 0$. Considering equality  $g_{22}=0$ once again, we can see that $\al_1 = \al_5^2 + \al_5 \al_8 + \la \al_8^2$.

\begin{enumerate}
\item Let $\la\neq0$.

\begin{enumerate}
\item Assume $\al_6\neq 0$. Then equality $g_{32}=0$ implies that $\al_5 = -\al_8 (\al_6 + \la \al_9)/\al_6$. Since  $\det(g)\neq0$, we have $\al_6^2+\al_6 \al_9+\la \al_9^2\neq0$ and $\al_8\neq0$.  Thus $g_{23}=0$ implies $\al_8 = -\al_6 / \la$. 

\item Assume $\al_6= 0$.  Since  $\det(g)\neq0$, we have $\al_5\neq0$ and $\al_9\neq0$. Consequently considering $g_{32}=0$ and $g_{23}=0$ we obtain that $\al_8 = 0$ and $\al_9 = \al_5$, respectively. 

\end{enumerate}

\item
Let $\la=0$. Since $\det(g)\neq0$, we have $\al_5+\al_8\neq0$ and $\al_5\neq0$. Consequently considering $g_{32}=0$ and $g_{23}=0$ we obtain $\al_6 = 0$ and $\al_9 = \al_5 + \al_8$, respectively. The required statement is proven.
 
\end{enumerate}

\item Assume $\algL=\algL_5$. Equality $g_{22}=0$ implies $\al_4 = 0$, $\al_7 = 0$, and then equality $g_{33}=0$ implies $\al_1 = \al_6^2 + \al_9^2$.

\begin{enumerate}
\item Let $\al_9\neq 0$. Equality $g_{23}=0$ implies that $\al_8 = -\al_5 \al_6/\al_9$. Since $\det(g)\neq0$, we have $\al_6^2+\al_9^2 \neq0$ and $g_{22}=0$ implies that $\al_9 = \pm \al_5$. 
 
\item Let $\al_9= 0$. Since $\det(g)=-\al_6^3 \al_8\neq0$, consequently considering $g_{23}=0$, $g_{22}=0$ we obtain $\al_5 = 0$ and $\al_8 = \pm \al_6$, respectively. The required statement is proven.

\end{enumerate}

\item Assume $\algL=\algL_6$. Consequently considering $g_{13}=0$, $g_{23}=0$ we obtain that $\al_8 = 0$ and $\al_7 = 0$, respectively. Considering equality $g_{13}=0$ once again, we can see that $\al_5 = \al_1 \al_9$ and $\al_2 = \al_4 \al_9$. Since $\det(g)=(\al_1^2 - \al_4^2) \al_9^2\neq0$, equality $g_{33}=0$ implies $\al_3 = \al_6 = 0$. 

If $\al_9\neq\pm 1$, then $g_{23}=0$ implies $\al_1=\al_4=0$; hence, $\det(g)=0$; a 
contradiction. Therefore, $\al_9 = \pm 1$. The required statement is proven.

\item Assume $\algL=\algL_7^{\la}$ for some $\la\in\FF$.
Considering $g_{13}=0$ we obtain $\al_8 = 0$, $\al_2 = \la \al_4 \al_9$ and $\al_5 = (\al_1 + \al_4) \al_9$. Note that $\al_9\neq0$, since $\det(g)\neq0$. Hence, it follows from $g_{33}=0$ that $\al_6 = -\al_3$.

\begin{enumerate}
\item Let $\la\neq 0$. Consequently considering $g_{23}=0$ and $g_{33}=0$, we obtain $\al_7 = 0$ and  $\al_3 = 0$, respectively.

\begin{enumerate}
\item Assume $\al_9 = -1$. Since the characteristic of the field is zero, $g_{23}=0$ implies $\al_4 = 0$ and $\al_1 = 0$. Thus, $\det(g)=0$; a contradiction.

\item Assume  $\al_9\neq \pm 1$. Then $g_{23}=0$ implies $\al_1 = -\al_4 \al_9/(1 + \al_9)$. Considering $g_{23}=0$ once again we obtain $\al_4(\al_9+\la (1+\al_9)^2)=0$. Thus $\det(g)=0$; a contradiction.

\end{enumerate}

Therefore, $\al_9 = 1$. 

\item Let $\la= 0$. Note that $\det(g)=(\al_1+\al_4) \al_9 (-\al_3 \al_7+\al_1 \al_9)$. Since $(\al_1+\al_4) \al_9 \neq0$, consequently considering $g_{11}=0$ and $g_{23}=0$ we can see that $\al_7 = 0$ and $\al_9 = 1$, respectively. The required statement is proven.

\end{enumerate}

\item Assume $\algL=\algL_8$ for some $\la\in\FF$. Equality $g_{13}=0$ implies that  $\al_8 = 0$, $\al_5 = \al_1 \al_9$, $\al_2 = \al_7 \al_9$. Hence, it follows from equality $g_{33}=0$ that $\al_7 = 0$, $\al_4 = \al_3 \al_9$, $\al_1 = \al_9^2$. The required statement is proven.

\item Assume $\algL=\algL_9$. Equality $g_{33}=0$ implies that $\al_4 = \al_3 \al_9$. Applying equality $g_{33}=0$ once again, we can see that $\al_7 = 0$ and $\al_1 = \al_9 (\al_3 + \al_9)$. Then equality $g_{13}=0$ implies $\al_8=0$. Thus $\al_9\neq0$, since $\det(g) = \al_9^2 (-\al_2 \al_3 + \al_5 (\al_3 + \al_9))$. Consequently considering $g_{23}=0$ and $g_{13}=0$ we obtain that $\al_2 = 0$ and $\al_5 = -\al_3 \al_9 + \al_3 \al_9^2 + \al_9^3$, respectively.

In case $\al_9 = -\al_3$ we have $\det(g)=0$; a contradiction. Therefore, $\al_9\neq -\al_3$ and $g_{13}=0$ implies that $\al_9=1$. The required statement is proven.

\red

\item  Assume $\algL=\algL_{10}$. Equality $g_{13}=0$ implies that $\al_7 = 0$. Applying equality $g_{23}=0$, we can see that $\al_8 = 0$. Note that $\al_9\neq0$, since $\det(g)\neq0$. Hence, it follows from $g_{33}=0$ that $\al_3 = \al_6=0$. 

\begin{enumerate}
\item Let $\al_9\neq 1$. Then $g_{13}=0$ implies $\al_1 = \al_4 = 0$; hence $\det(g)=0$, a contradiction. 

\item Let $\al_9=1$. Then the required statement holds.
\end{enumerate}

\item Assume $\algL=\algL_{11}$. Equality $g_{11}=0$ implies that $\al_7 = 0$. Applying equality $g_{22}=0$, we can see that $\al_8 = 0$. Hence, it follows from equality $g_{33}=0$ that $\al_4=0$ and $\al_1=\al_9^2$.  The required statement is proven.

\end{enumerate}
\end{proof}

\begin{cor}\label{cor_aut}
Given $\algL\in {\bf L}_3$, the group $\Aut(\algL)$ has dimension {\red  $2\leqslant d\leqslant 5$} as an affine variety, where each of the possibilities holds for some algebra $\algL\in {\bf L}_3$. 
\end{cor}
{\red
\begin{proof}
For each $\algL\in {\bf L}_3$, the group $\Aut(\algL)$ is the following affine variety: 
$$\left \{  \Big(b_1,\ldots,b_d,f_1(\un{b}),\ldots,f_r(\un{b}) \Big)\,\Big|\, b_1,\ldots,b_d\in\FF \text{ and } h_1(\un{b})\neq0,\ldots, h_s(\un{b})\neq0 \right\}$$
for some polynomials $f_1,\ldots,f_r,h_1,\ldots,h_s$ from $\FF[x_1,\ldots,x_d]$ with $r,s\geqslant0$, where $\un{b}$ stands for $(b_1,\ldots,b_d)$. Then the dimension of $\Aut(\algL)$ is $d$. Applying this reasoning together with Theorem~\ref{theo_aut}, we complete the proof.
\end{proof}
}

An automorphism $\de$ of an algebra $\algA$ is called {\it diagonal} if there exists some basis $\{e_1,\ldots,e_n\}$ of $\algA$ such that $\varphi(e_i)=\al_i e_i$ for some $\al_i\in\FF$. For short, we write a diagonal $n\times n$ matrix with diagonal elements $\al_1,\ldots,\al_n\in\FF$ as $\diag{\al_1,\ldots,\al_n}$.  The importance of the diagonal automorphisms for invariant theory follows from the next remark.

\begin{remark}\label{remark_diag}
If there are $a_1,\ldots,a_n\in\NN$ such that $\de=\diag{\al^{a_1},\ldots,\al^{a_n}}$ is an automorphism of $\algA$ for all $\al\in\FF$ and $a_{i_0}>0$ for some $1\leqslant i_0\leqslant n$, then the algebra of invariants $I_m(\algL)$ lies in $\FF[x_{ri}\,|\,1\leqslant r\leqslant m,\; 1\leqslant i\leqslant n,\; i\neq i_0]$. {\red As an example, if $\diag{\al,1,\ldots,1}$ is an automorphism of $\algA$ for all $\al\in\FF$, then $I_m(\algL)$ lies in $\FF[x_{ri}\,|\,2\leqslant r\leqslant m,\; 1\leqslant i\leqslant n ]$.}
\end{remark}

Theorem~\ref{theo_aut} implies the following lemma.

\begin{lemma}\label{lemma_aut_diag}
Here is a list of some diagonal automorphisms $\de$:
\begin{longtable}{lcllcl}
$\algL_1$&$: $&$ \de=\diag{\al^2, \al, 1},$  &
$\algL_2^{{\red \la\neq0}}$&$: $&$ \de=\diag{\al, \be, 1},$\\ 
$\algL_3$&$: $&$ \de=\diag{1, \al, 1},$ &  
$\algL_4^{\la}$&$: $& $\de=\diag{\al^2, \al, \al},$\\ 
$\algL_5$&$: $&$ \de=\diag{\al^2, \al, \pm \al},$ & 
$\algL_6$&$: $&$ \de=\diag{\al,\al, 1},$\\  
$\algL_7^{\la}$&$: $&$ \de=\diag{\al, \al, 1},$&  
$\algL_8$&$: $&$ \de=\diag{\al^2, \al^3, \al},$\\   
\red $\algL_{10}$&$: $& \red  $ \de=\diag{\al, \be, 1},$&  
\red $\algL_{11}$&$: $& \red  $ \de=\diag{\al^2, \be, \al},$ \\ 
\end{longtable}
for all $\la\in\FF$ and all non-zero $\al,\be\in\FF$.
\end{lemma}

\section{Generators for algebras of invariants}\label{section_invariants} 

In this section we describe the algebra of polynomial invariants $I_m(\algL)$  
for any $\algL\in{\bf L}_3$. For short, we denote $x_r:=x_{r1}$, $y_r:=x_{r2}$, $z_r:=x_{r3}$ for all $1\leqslant r\leqslant m$. Given $\un{r}=(r_1,\ldots,r_k)\in\{1,\ldots,m\}^k$, we write $x_{\un{r}}=x_{r_1}\cdots x_{r_k}$, $y_{\un{r}}=y_{r_1}\cdots y_{r_k}$, $z_{\un{r}}=z_{r_1}\cdots z_{r_k}$.

Note that by formula~\eqref{eq_action} an automorphism $g=(g_{ij})_{1\leqslant i,j\leqslant 3}$ of $\Aut(\algL)\leqslant \GL_3$ acts on $\FF[\algL^m]$ as follows:
\begin{eq}\label{eq_action2} 
\begin{array}{clc}
g^{-1} x_r &=& g_{11} x_r + g_{12} y_r + g_{13} z_r, \\
g^{-1} y_r &=& g_{21} x_r + g_{22} y_r + g_{23} z_r, \\
g^{-1} z_r &=& g_{31} x_r + g_{32} y_r + g_{33} z_r
\end{array}
\end{eq}%
for all $1\leqslant r\leqslant m$.

\begin{theo}\label{theo_main}
For every $\la\in\FF$ and $m>0$ we have the following equalities for algebras of invariants:
\begin{enumerate}
\item[$\bullet$] $I_m(\algL_1)$, $I_m(\algL_2^{{\red\la\neq0}})$, $I_m(\algL_7)$, {\red $I_m(\algL_{10})$} are all equal to $\FF[z_1,\ldots,z_m]$,  

\item[$\bullet$] $I_m(\algL_3)$ is generated by $1, z_1, \ldots, z_m$ and $x_r z_s - z_r x_s $ $(1 \leqslant r < s \leqslant m$),

\item[$\bullet$] $I_m(\algL_4^{\la})=I_m(\algL_5)=I_m(\algL_8)={\red I_m(\algL_{11})}=\FF$,  
\item[$\bullet$] $I_m(\algL_6)$ is generated by $1, z_r z_s$ for all $1\leqslant i\leqslant j\leqslant m$,  

\item[$\bullet$] $I_m(\algL_9)$ is generated by $1, z_1, \ldots, z_m$ and $(x_r-y_r) z_s - z_r (x_s-y_s)$ for all $1\leqslant r < s \leqslant m$.
\end{enumerate}
\end{theo}
\begin{proof} Lemma~\ref{lemma_aut_diag} together with Remark~\ref{remark_diag} implies that for every $\la\in\FF$ and $m>0$ we have the following inclusions and equalities for algebras of invariants:
\begin{enumerate}
\item[$\bullet$] $I_m(\algL_1)$, $I_m(\algL_2^{{\red\la\neq0}})$, $I_m(\algL_6)$, $I_m(\algL_7)$, {\red $I_m(\algL_{10})$}   belong to $\FF[z_1,\ldots,z_m]$,  

\item[$\bullet$] $I_m(\algL_3)\subset\FF[x_1,\ldots,x_m,z_1,\ldots,z_m]$,  

\item[$\bullet$] $I_m(\algL_4^{\la})=I_m(\algL_5)=I_m(\algL_8)={\red I_m(\algL_{11})}=\FF$.  
\end{enumerate}

Theorem~\ref{theo_aut} implies that $z_1,\ldots, z_m\in I_m(\algL)$ for every $\algL$ from the following list: $\algL_1$, $\algL_2^{{\red\la\neq0}}$, $\algL_3$, $\algL_7$, $\algL_9$, {\red $\algL_{10}$}. Thus, the statement of the theorem holds for algebras   $\algL_1$, $\algL_2^{{\red\la\neq0}}$,  $\algL_4^{\la}$, $\algL_5$, $\algL_7$, $\algL_8$, {\red $\algL_{10}$, $\algL_{11}$}.

\medskip
Consider the algebra $\algL_3$. The subspace $V=\FF\text{-span}\{e_1,e_3\}$ is invariant with respect to $\Aut(\algL_3)$ and the restriction of $\Aut(\algL_3)$ to $V$ is the set of all matrices
$$\matr{1}{\al}{0}{1}$$
for all $\al\in\FF$. Then applying Lemma~\ref{lemma_dim2} we prove the statement about $\algL_3$.

\medskip
Consider the algebra $\algL_6$. Note that $z_r z_s \in I_m(\algL_6)$ for all $1\leqslant r\leqslant s\leqslant m$ by Theorem~\ref{theo_aut}. Then a monomial in $z_1,\ldots,z_m$ lies in $I_m(\algL_6)$ if and only if it has even degree. Restricting the group $\Aut(\algL_6)$ to $\FF\text{-span}\{e_3\}$ and applying Lemma~\ref{lemma_diag} we conclude the proof about $\algL_6$. 

\medskip
Consider the algebra $\algL_9$. Using formlas~(\ref{eq_action2}) we can see that $(x_r-y_r) z_s - z_r (x_s-y_s)$ lies in $I_m(\algL_9)$. For every $\al\in\FF\backslash\{-1\}$ consider an automorphism 
$$g=\matrThree{1 + \al}{0}{\al}{\al}{1}{0}{0}{0}{1}  $$
of $\algL_9$  (see Theorem~\ref{theo_aut}). For $e_1'=e_1-e_2$ we have that 
\begin{eq}\label{eq_gnew}
ge_1' = e_1' + \al e_3,\;\; g e_2 = \al e'_1 + (\al+1) e_2,\;\; g e_3 = e_3.
\end{eq}
\noindent{}For $x_r'=x_r-y_r$ we have that $\FF[\algL^m]=\FF[x'_r,y_r,z_r\,|\,1\leqslant r\leqslant m]$, and until the end of this proof we consider elements of $\FF[\algL^m]$ as linear combinations of monomials in $x'_r$, $y_r$, $z_r$.

For each $m\geqslant0$, denote by $\Omega_m$ the set of all monomials $w=w_1\cdots w_m \in \FF[\algL^m]$ of multidegree $1^m$ for some $w_r\in\{x'_r,y_r,z_r\}$, where $1\leqslant r\leqslant m$.  Given $w\in\Omega_m$, denote 
$$\begin{array}{rcl}
\deg_{x'}(w) & = & \deg_{x'_1}(w) + \cdots + \deg_{x_m}(w), \\
\deg_{y}(w) & = & \deg_{y_1}(w) + \cdots + \deg_{y_m}(w), \\
\deg_{z}(w) & = & \deg_{z_1}(w) + \cdots + \deg_{z_m}(w). \\
\end{array}$$ 
As an example, for $m=3$ we have $\deg_y(z_1 x'_2y_3)=1$ and $\deg_y(x'_1y_2y_3) =2$. Introduce a complete order on $\Omega_m$ as follows: for $v=v_1\cdots v_m$ and $w=w_1\cdots w_m$ from $\Omega_m$ we have 
$v< w$ if and only if 
\begin{enumerate}
\item[(a)] $\deg_y(v)<\deg_y(w)$, or 

\item[(b)] $\deg_y(v)=\deg_y(w)$ and $\deg_x(v)<\deg_x(w)$, or 

\item[(c)] $\deg_y(v)=\deg_y(w)$, $\deg_x(v)=\deg_x(w)$, and there exists $1\leqslant r\leqslant m$ with
\begin{enumerate}
\item[$\bullet$] $v_1=w_1, \ldots, v_{r-1}=w_{r-1}$,

\item[$\bullet$] $(v_r,w_r)$ is one of the following pairs: $(x'_r,y_r)$, $(z_r,y_r)$, $(z_r,x'_r)$. 
\end{enumerate}
\end{enumerate}
As an example, for $m=3$ we have 
$$x_1' x_2' y_3 < z_1 y_2 y_3,\;\;\;\; x_1' x_2' y_3 < z_1 x'_2 y_3,\;\; \text{ and }\;\; y_1 z_2 x'_3 < y_1 x'_2 y_3. $$


\noindent{}For a multilinear $f=\sum_j \be_j w_j\in\FF[\algL^m]$ with non-zero $\be_j\in\FF$ and $w_j\in\Omega_m$ we say that $w_{j_0}$ is the highest term for $f$, if there is no $j\neq j_0$ with $w_j>w_{j_0}$. Note that the highest term for $f$ always exists and is unique.   

Consider a multilinear $f=\be w + \sum_j \be_j w_j\in I_m(\algL)$ with non-zero $\be,\be_j\in\FF$ and $w,w_j\in\Omega_m$, where $w$ is the highest term for $f$. Without loss of generality we can assume that  
$$w=x'_1\cdots x'_r y_{r+1} \cdots y_{s} z_{s+1} \cdots z_m,$$
where $0\leqslant r\leqslant s\leqslant m$. Formulas~(\ref{eq_action2}) and~(\ref{eq_gnew}) and imply that  

$$g^{-1}f = \be (x_1' + \al z_1)\cdots (x_r' + \al z_r) 
\big(\al x'_{r+1} + (\al+1) y_{r+1}\big)\cdots \big(\al x'_s + (\al+1) y_s\big) z_{s+1} \cdots z_m 
+ \sum_j \be_j g^{-1}w_j.$$
It is easy to see that $w$ is also the highest term of $g^{-1}f$ with the coefficient $(\al+1)^{s-r}$. Therefore, $(\al+1)^{s-r}=1$ for all $\al\in\FF\backslash\{-1\}$. Hence, $s=r$, i.e., $w$ does not contain letters $y_1,\ldots,y_m$. Hence, $f\in \FF[x'_r,z_r\,|\,1\leqslant r\leqslant m]$. 

Note that the subspace $V=\FF\text{-span}\{e'_1,e_3\}$ is invariant with respect to $g$ and the restriction of $g$ to $V$ is
$$\matr{1}{\al}{0}{1}$$
for all $\al\in\FF$. Since the coordinate ring of $V$ is $\FF[x'_r,z_r\,|\,1\leqslant r\leqslant m]$, applying Lemma~\ref{lemma_dim2} we conclude the proof of the statement about $\algL_9$.

\end{proof}

\section{Trace invariants}\label{section_inv}

Assume than $n=3$. Writing down $M_{ij}$ as $M_{ij}=(M_{ij1},M_{ij2},M_{ij3})$, we can see that formula~(\ref{eq_trL_trR}) implies that 
\begin{eq}\label{eq_trL3}
\tr(\chi_r\chi_0)= \left(\sum_{i=1}^3\al_i\right) x_r + \left(\sum_{i=1}^3\be_i\right) y_r + \left(\sum_{i=1}^3\ga_i\right) z_r 
\;\text{ for }\;
M=\matrThree{(\al_1,\ast,\ast)}{(\ast,\al_2,\ast)}{(\ast,\ast,\al_3)}{(\be_1,\ast,\ast)}{(\ast,\be_2,\ast)}{(\ast,\ast,\be_3)}{(\ga_1,\ast,\ast)}{(\ast,\ga_2,\ast)}{(\ast,\ast,\ga_3)},
\end{eq}
\begin{eq}\label{eq_trR3}
\tr(\chi_0\chi_r)= \left(\sum_{i=1}^3\al_i\right) x_r + \left(\sum_{i=1}^3\be_i\right) y_r + \left(\sum_{i=1}^3\ga_i\right) z_r 
\;\text{ for }\;
M=\matrThree{(\al_1,\ast,\ast)}{(\be_1,\ast,\ast)}{(\ga_1,\ast,\ast)}{(\ast,\al_2,\ast)}{(\ast,\be_2,\ast)}{(\ast,\ga_2,\ast)}{(\ast,\ast,\al_3)}{(\ast,\ast,\be_3)}{(\ast,\ast,\ga_3)}.
\end{eq}

\medskip
\begin{prop}\label{prop_tr}
For every $\la\in\FF$, $m>0$ and $1\leqslant r,s\leqslant m$ we have the following trace formulas:
\begin{longtable}{l|ll}
\hline
$\algL_1$ & $\tr(\chi_r \chi_0) = z_r,$ &  $\tr(\chi_0 \chi_r)= - 3z_r,$    \\  
          & $\tr(\chi_r (\chi_s \chi_0)) = z_r z_s,$ &  $\tr((\chi_s\chi_0) \chi_r)= -z_r z_s,$\\
         & $\tr(\chi_r(\chi_0 \chi_s))=  -z_r z_s,$     &$\tr((\chi_0\chi_s) \chi_r)= 5 z_r z_s.$\\
\hline   
$\algL_2^{{\red\la\neq0}}$ & $\tr(\chi_r \chi_0) =z_r, $ &$ \tr(\chi_0 \chi_r)=(\la-1)z_r,$ \\ 
           &$\tr(\chi_r (\chi_s \chi_0)) = z_r z_s,$ &  $\tr((\chi_s \chi_0) \chi_r)= -z_r z_s,$\\
          &$\tr(\chi_r(\chi_0 \chi_s))= -z_r z_s,    $&$\tr((\chi_0\chi_s) \chi_r)= (1+\la^2) z_r z_s.$ \\ 
\hline
$\algL_3$ & $\tr(\chi_r \chi_0) = z_r,  $&$ \tr(\chi_0 \chi_r)=-z_r,$  \\
         &$ \tr(\chi_r (\chi_s \chi_0)) = z_r z_s,$ &  $\tr((\chi_s\chi_0) \chi_r)= -z_r z_s,$\\
          &$\tr(\chi_r(\chi_0 \chi_s))=  -z_r z_s,$    &$\tr((\chi_0\chi_s) \chi_r)=  z_r z_s,$  \\ 
         &$\tr((\chi_r\chi_s) \chi_0))=  0,$    &$\tr(\chi_0(\chi_r \chi_s))=  0.$  \\ 
\hline
$\algL_4^{\la}$  &  $\tr(\chi_r \chi_0) = 0,  $&$ \tr(\chi_0 \chi_r)=0,$\\
          &$ \tr(\chi_r (\chi_s \chi_0)) = 0, $&$  \tr((\chi_s\chi_0) \chi_r)=0, $\\
          &$\tr(\chi_r(\chi_0 \chi_s))=  0,   $&$\tr((\chi_0\chi_s) \chi_r)=0.$ \\ 
\hline
$\algL_5$ &  $\tr(\chi_r \chi_0) =0,  $&$  \tr(\chi_0 \chi_r)=0,$ \\
         &$ \tr(\chi_r (\chi_s \chi_0)) = 0, $&$  \tr((\chi_s \chi_0) \chi_r)= 0,$\\
           & $\tr(\chi_r(\chi_0 \chi_s))=  0,   $&$\tr((\chi_0\chi_s) \chi_r)= 0.$ \\
\hline
$\algL_6$& $\tr(\chi_r \chi_0) =0 ,  $&$ \tr(\chi_0 \chi_r)=0,$\\ 
         & $\tr(\chi_r (\chi_s \chi_0)) = 0, $&$  \tr((\chi_s\chi_0) \chi_r)=0, $\\
          &$\tr(\chi_r(\chi_0 \chi_s))=  0,   $&$\tr((\chi_0\chi_s) \chi_r)=  2 z_r z_s.$\\
\hline
$\algL_7^{\la}$ &  $\tr(\chi_r \chi_0) =0,  $&$  \tr(\chi_0 \chi_r)=z_r,$\\ 
          &$  \tr(\chi_r (\chi_s \chi_0)) = 0, $&$  \tr((\chi_s \chi_0) \chi_r)= 0,$\\
          &$\tr(\chi_r(\chi_0 \chi_s))=  0,    $&$\tr((\chi_0\chi_s) \chi_r)=  (1+2 \la) z_r z_s.$\\
\hline
$\algL_8$ &$ \tr(\chi_r \chi_0) = 0,  $&$ \tr(\chi_0 \chi_r)=0,$\\ 
          &$ \tr(\chi_r (\chi_s \chi_0)) = 0, $&$   \tr((\chi_s \chi_0) \chi_r)= 0,$\\
          &$\tr(\chi_r(\chi_0 \chi_s))=  0,   $&$\tr((\chi_0\chi_s) \chi_r)=  0. $\\
\hline
$\algL_9$ &$\tr(\chi_r \chi_0) =0,  $& $\tr(\chi_0 \chi_r)=z_r,$ \\ 
          &$ \tr(\chi_r (\chi_s \chi_0)) = 0,$&$  \tr((\chi_s\chi_0) \chi_r)=0, $\\
          &$\tr(\chi_r(\chi_0 \chi_s))=  0, $&$\tr((\chi_0\chi_s) \chi_r)= z_r z_s, $ \\
          &$\tr((\chi_r\chi_s) \chi_0))=  0,$    &$\tr(\chi_0(\chi_r \chi_s))=  0.$  \\ 
\hline
\red $\algL_{10}$ & \red $\tr(\chi_r \chi_0) = 0,  $& \red  $\tr(\chi_0 \chi_r)=2 z_r,$\\ 
          & \red $\tr(\chi_r (\chi_s \chi_0)) = 0, $& \red  $   \tr((\chi_s \chi_0) \chi_r)= 0,$\\
          & \red  $\tr(\chi_r(\chi_0 \chi_s))=  0,   $& \red $\tr((\chi_0\chi_s) \chi_r)=  2 z_r z_s. $\\
\hline
\red $\algL_{11}$ & \red $ \tr(\chi_r \chi_0) = 0,  $& \red $ \tr(\chi_0 \chi_r)=0,$\\ 
          & \red $ \tr(\chi_r (\chi_s \chi_0)) = 0, $& \red $   \tr((\chi_s \chi_0) \chi_r)= 0,$\\
          & \red $\tr(\chi_r(\chi_0 \chi_s))=  0,   $& \red $\tr((\chi_0\chi_s) \chi_r)=  0. $\\
\hline      
\end{longtable}
 
\end{prop}
\begin{proof} We take tableaux of multiplication for Leibniz algebras from Theorem~\ref{theo_Leib3}. To calculate traces $\tr(\chi_r \chi_0)$ and $\tr(\chi_0 \chi_r)$ we apply equalities (\ref{eq_trL3}) and (\ref{eq_trR3}), respectively. To calculate the rest of the traces we apply Proposition~\ref{prop_trace}. 
\end{proof}

\begin{cor}\label{cor_API}
Artin--Procesi--Iltyakov Equality does not hold for $\algL \in{\bf L}_3$ if and only if $\algL$ is isomorphic to $\algL_3$ or $\algL_9$.
\end{cor}
\begin{proof} We apply Proposition~\ref{prop_trace}, Theorem~\ref{theo_main}, part (c) of Remark~\ref{remark_key} and the fact that the algebra $I_m(\algL)$ has $\NN^m$-grading by multidegrees.
\end{proof}

\begin{prop}\label{prop_nilp} We have 
\begin{enumerate}
\item[$\bullet$] $\ncl(\algL_4^{\la})=\ncl(\algL_5)={\red \ncl(\algL_{11})}=3$ for all $\la\in\FF$; 

\item[$\bullet$] $\ncl(\algL_8)=4$; 

\item[$\bullet$] the algebras $\algL_1$, $\algL_2^{{\red\la\neq0}}$, $\algL_3$, $\algL_6$, $\algL_7^{\la}$, $\algL_9$, {\red $\algL_{10}$} are not nilpotent  for all $\la\in\FF$.
\end{enumerate}
\end{prop}
\begin{proof} For the algebras $\algL_1$, $\algL_2^{{\red\la\neq0}}$, $\algL_3$, $\algL_6$, $\algL_7^{\la}$ the elements $(\cdots ((e_2 e_3) e_3) \cdots )e_3$ are non-zero. Similarly, $(\cdots ((e_1 e_3) e_3) \cdots )e_3\neq0$ in $\algL_9$ and in {\red $\algL_{10}$}. Since $(e_3 e_3)e_3=e_2\neq0$ in $\algL_8$, we can see that  $\ncl(\algL_8)=4$. The cases of $\algL_4^{\la}$, $\algL_5$ and  {\red $\algL_{11}$} are trivial.
\end{proof}

\begin{cor}\label{cor_nilp}
For an algebra $\algL\in{\bf L}_3$, the following four conditions are equivalent:
\begin{enumerate}
\item[(a)] the algebra of invariants $I_m(\algL)$ is $\FF$ for all $m\geqslant1$;

\item[(b)] the algebra of invariants $I_1(\algL)$ is $\FF$;

\item[(c)] the algebra $\algL$ is nilpotent;

\item[(d)] $\tr(\chi_0 \chi_1)=\tr((\chi_0 \chi_1) \chi_1) = 0$ for $\algL$.
\end{enumerate}
\end{cor}
\begin{proof}
The claim of the corollary follows from Proposition~\ref{prop_nilp}, Theorem~\ref{theo_main} and Proposition~\ref{prop_tr}.
\end{proof}

\begin{cor}\label{cor_separ}
Two non-nilpotent algebras from ${\bf L}_3$ can be distinguished by means of the traces of degrees $\leqslant2$ and the dimensions of its groups of automorphisms. That is, given $\algL\in{\bf L}_3$, denote $d=\dim\Aut(\algL)$. Then 
 
\begin{longtable}{lllllll}
$\algL  \simeq  \algL_1 $&$ \Longleftrightarrow  $&$ \tr(\chi_1\chi_0)\neq0, $&$ d=2.$    \\

$\algL  \simeq  \algL_2^{\la\neq 0} $&$ \Longleftrightarrow  $&$ \tr(\chi_1\chi_0)\neq0,  $&$ d=3, $&$ \tr(\chi_0\chi_1)=(\la-1)\tr(\chi_1\chi_0).$  \\


$\algL  \simeq \algL_3 $&$ \Longleftrightarrow  $&$ \tr(\chi_1\chi_0)\neq0, $&$ d=3,  $&$  \tr(\chi_0\chi_1)=-\tr(\chi_1\chi_0).$  \\

$\algL \simeq \algL_6 $&$ \Longleftrightarrow  $&$  \tr(\chi_1\chi_0)=0, $&$ d=2, $&$\tr(\chi_0\chi_1)=0. $  \\

$\algL  \simeq  \algL_7^{\la\neq 0} $&$ \Longleftrightarrow  $&$ \tr(\chi_1\chi_0)=0, $&$ d= 2, $&$ \tr(\chi_0\chi_1)\neq0,\;\; \tr((\chi_0\chi_1)\chi_1)=(1+2\la)\tr(\chi_0\chi_1)^2.$ \\

$\algL  \simeq  \algL_7^{0} $&$ \Longleftrightarrow  $&$\tr(\chi_1\chi_0)=0, $&$ d=3.  $   \\ 

$\algL  \simeq  \algL_9 $&$ \Longleftrightarrow  $&$ \tr(\chi_1\chi_0)=0, $&$ d=2,  $&$ \tr(\chi_0\chi_1)\neq0, \;\; \tr((\chi_0\chi_1)\chi_1)=\tr(\chi_0\chi_1)^2.$  \\ 

\red $\algL  \simeq  \algL_{10} $& \red $ \Longleftrightarrow  $& \red $\tr(\chi_1\chi_0)=0, $& \red $ d=4.  $   \\ 
\end{longtable}
 \end{cor}

\section*{Acknowledgement}

We would like to express our sincere gratitude to the anonymous referee for the valuable and constructive comments, which have helped us improve the quality and clarity of our paper.

%
%
%
%
%
%
%


\bibliographystyle{abbrvurl}
\bibliography{literature}

\end{document}